%% file: MS.tex
\documentclass[11pt]{article}
\usepackage{amsfonts}
\usepackage{amsthm}
\usepackage{amssymb}
\usepackage{graphicx, enumerate}
\usepackage{amsmath}
\usepackage{latexsym}
\usepackage{longtable}
\usepackage{tabularx}
\usepackage{amsmath}
\usepackage{amsfonts}
\usepackage{amsthm}
\usepackage{setspace}
\usepackage{graphicx}
\usepackage{float}
\usepackage{rotating}
\usepackage{tikz}
\usepackage{verbatim}

\usepackage{caption}
\usepackage{subcaption}

\usepackage{soul}

\input{comments_input}

\usepackage{tikzsymbols}
\usepackage{geometry}
\newcommand{\ka}{\mathrm{KA}}
\newcommand{\gka}{\Gamma\text{-KA}}

\newcommand{\ms}{\mathrm{MS}_}

\newcommand{\gms}{\mathrm{MS}_{\Gamma}}

\newtheorem{theorem}{Theorem}[section]
\newtheorem{claim}[theorem]{Claim}
\newtheorem{lemma}[theorem]{Lemma}
\newtheorem*{lemma*}{Lemma}

\newtheorem{oprb}[theorem]{Open Problem}
\newtheorem{obs}[theorem]{Observation}

\theoremstyle{definition}
\newtheorem{definition}[theorem]{Definition}

\theoremstyle{definition}

\newtheorem*{prb*}{Open Problem}

\newtheorem{exm}[theorem]{Example}

\def\gr{\mathcal{G}}

\vspace{5cm}

\begin{document}


\title
{\Large \sc \bf {Magic  squares on Abelian groups}}
\date{}
\author{{{Sylwia Cichacz$^{1}$, Dalibor Froncek$^{2}$}}\\
\normalsize $^1$AGH University of Krakow, Poland, cichacz@agh.edu.pl\\
\normalsize $^2$University of Minnesota Duluth, U.S.A.,  dalibor@d.umn.edu
}

\maketitle

\begin{abstract}
	Let $(\Gamma,+)$ be an Abelian group of order $n^2$ and $\gms(n)$ be an $n\times n$ array whose entries are all elements of $\Gamma$. Then $\gms(n)$ is a $\Gamma$-magic square if all row, column, main and backward main diagonal sums are equal to the same element $\mu\in\Gamma$.
We prove that for every Abelian group $\Gamma$ of order $n^2$, $n>2$, there exists a magic square $\gms(n)$ where the square entries are elements of $\Gamma$.
\end{abstract}

\noindent
\textbf{Keywords:}  Magic squares, magic rectangles, Latin squares, Abelian group, Kotzig array

\noindent
\textbf{2000 Mathematics Subject Classification:} 05B15

\input{10_intro_22d}
\input{20_defs_and_related_22g}

\input{30_tools_22b}

\input{40_prelim_results_22g}

\input{50_construction_n_odd_22a}
\input{60_construction_n=2_s_22a}

\input{70_construction_n_even_22e}

\input{80_main_22h}
\input{90_conclusion_22d}

\input{99_references_22a}
\end{document}

%% file: comments_input.tex
\usepackage{mathtools}
\usepackage{amsmath,amssymb,amsbsy}
\usepackage{graphicx}
\usepackage{multirow}
\usepackage{amsfonts}
\usepackage{amsthm}
\usepackage{amssymb}
\usepackage{graphicx, enumerate}
\usepackage{amsmath}
\usepackage{latexsym}
\usepackage{longtable}
\usepackage{tabularx}
\usepackage{amsmath}
\usepackage{amsfonts}
\usepackage{amsthm}
\usepackage{setspace}
\usepackage{graphicx}
\usepackage{float}
\usepackage{rotating}
\usepackage{tikz}
\usepackage{verbatim}
\usepackage{soul}
\usepackage{xcolor}
\usepackage{cite}
\usepackage{changepage}

\usepackage[normalem]{ulem}

\usepackage[main=english,czech,slovak]{babel}

\definecolor{darkpastelgreen}{rgb}{0.01, 0.75, 0.24}
\definecolor{blue-violet}{rgb}{0.54, 0.17, 0.89}

\newcommand{\mmag}[1]{{\color{magenta}{#1}}}
\newcommand{\bv}[1]{{\color{blue-violet}{#1}}}

\usepackage{soul}
\usepackage{ulem}


%% file: 10_intro_22d.tex
\section{Introduction}\label{sec:intro}

 A \emph{magic square} of order $n$ denoted by MS$(n)$ is an $n\times n$ array with entries $1,2,\dots, n^2$ such that the sum of each row, column, and the main forward and backward diagonal is the same \emph{magic constant} $c=n(n^2+1)/2$. When we only require the row and column sums to be equal and disregard the diagonals, we speak about a \emph{semi-magic square.}  Magic squares are one of the oldest mathematical structures. The earliest known magic square is a $3\times3$ magic square called \textit{Lo Shu magic square} and can be traced in Chinese literature as far back as 2800 B.C. Since then, certainly, many people have studied magic squares. For a  survey of magic squares, see Chapter 34 in \cite{handbook}.

A broader concept is an $m\times n$ \emph{magic rectangle} with entries $1,2,\dots mn$ where all row sums are equal to the \emph{row constant} $r=n(mn+1)/2$ all column sums are equal to the \emph{column constant} $c=m(mn+1)/2$. A semi-magic square is then an $n\times n$ magic rectangle.

Magic squares and rectangles can be generalized in may different ways. For instance, we may require that the entries are all primes or integers with some other special properties, or elements of an Abelian group $\Gamma$ or order $n^2$. We call the latter square a \emph{$\Gamma$-magic square} and denote it by $\gms(n)$. A formal definition is given in Section~\ref{sec:defs and related}.

To avoid confusion with the order of the group forming a magic square, we will from now on refer to the order of MS$_{\Gamma}(n)$  as the \emph{side}. So, MS$_{\Gamma}(n)$  will be called a \emph{magic square of side $n$}.

To our knowledge, there were so far two attempts at constructing $\Gamma$-magic squares. Evans~\cite{Evans} characterized the existence of magic rectangles on cyclic groups (he called them \emph{modular magic squares}).

Sun and Yihui~\cite{Sun-Yihui} claimed to have found (without proof) $\Gamma$-magic squares for any $\Gamma$. They in fact only presented a construction for a $Z_n\oplus Z_n$-magic square $\ms{Z_n\oplus Z_n}(n)$ for $n$ odd.

In this paper, we prove that there exists a magic square of side $n$ for every $n>2$ with elements of any Abelian group with $n^2$ elements.

	The paper is organized as follows. 

	In Section~\ref{sec:defs and related}, we give definitions and an overview of known results.  In Section~\ref{sec:LS and KA} we introduce key tools such as Latin squares, Kotzig arrays, and $\Gamma$-Kotzig arrays.  Section~\ref{sec:prelim}  contains preliminary lemmas and observations needed to build $\Gamma$-magic squares. In Section~\ref{sec: const}, we present the main constructions and prove that $\Gamma$-magic squares exist for all Abelian groups $\Gamma$ of order $n^2$, for all $n>2$.  We conclude the paper with some final remarks and open questions in Section~\ref{sec:conclusion}.


%% file: 20_defs_and_related_22g.tex
\section{Definitions and related results}\label{sec:defs and related}

We start with a definition of a more general structure, $\Gamma$-magic rectangles.

\begin{definition}\label{def:abel-magic-rect}
	Let $(\Gamma,+)$ be an Abelian group of order $mn$ and MR$_{\Gamma}(m,n)$ be an $m\times n$ array whose entries are all elements of $\Gamma$. Then MR$_{\Gamma}(m,n)$ is a $\Gamma$-magic rectangle if all row sums are equal to some element $\rho\in\Gamma$ and all column sums are equal to some element $\sigma\in\Gamma$.
\end{definition}

Setting $m=n$ and adding conditions on diagonal sums, we obtain a definition of a $\Gamma$-magic square.

\begin{definition}\label{def:abel-magic-square}
	Let $(\Gamma,+)$ be an Abelian group of order $n^2$ and $\gms(n)$ be an $n\times n$ array whose entries are all elements of $\Gamma$. Then $\gms(n)$ is a $\Gamma$-magic square if all row, column, main and backward main diagonal sums are equal to the same element $\mu\in\Gamma$.
\end{definition}

To avoid confusion with the order of the group forming a magic  square, we will from now on refer to the order of MS$_{\Gamma}(n)$  as the \emph{side}. So, MS$_{\Gamma}(n)$  will be called a magic square of \emph{side} $n$.

Evans~\cite{Evans} characterized the existence of magic rectangles on cyclic groups (he called them \emph{modular magic squares}) by observing the following.

\begin{theorem}[Evans~\cite{Evans}]\label{thm:modular-magic-sq}
	There exists a non-trivial $Z_{mn}$-magic rectangle $\mathrm{MR}_{Z_{mn}}(m,n)$ if and only if $m>1,n>1$ and $m\equiv n\pmod2$.
\end{theorem}

One can notice that the characterization is exactly the same as for magic rectangles in integers, except the case of $m=n=2$ when a magic rectangle does not exist while a $Z_4$-magic $2\times2$ rectangle does (and is shown in Figure~\ref{fig: MR_Z_4}).

\begin{figure}[H]
	$$
		\begin{tabular}{|c|c|}
		\hline
		 0&1  \\
		\hline
		 3&2  \\
		\hline
		\end{tabular}
	$$
\caption{MR$_{Z_{4}}(2)$}
\label{fig: MR_Z_4}	
\end{figure}

{
Another example of a $Z_{n^2}$-magic rectangle which is not a $Z_{n^2}$-magic square is shown in Figure~\ref{fig: MR_Z_36}, 	where $n=4, \rho=6$ and $\sigma=14$.

\begin{figure}[H]
	$$
		\begin{tabular}{|c|c|c|c|}
		\hline
		  0& 1& 2& 3\\
		\hline
		  7& 6& 5& 4\\
		\hline
		  8& 9&10&11\\
		\hline
		 15&14&13&12\\
		\hline
		\end{tabular}
	$$
\caption{$\ms{Z_{36}}(6)$}
\label{fig: MR_Z_36}	
\end{figure}
}

\noindent

The identity element of a group $\Gamma$ will be denoted by $0$. Recall that an element $\iota\in\Gamma$ of order 2 (i.e., $\iota\neq 0$ and $2\iota=0$) is called an \emph{involution}. 
For convenience, let $\gr$ denote the set consisting of all Abelian groups which are
of odd order (and thus have no involution) or contain more than one involution and recall that $Z_{2n}\notin\mathcal{G}$ because it has a unique involution $\iota=n$.

Cichacz and Hinc generalized Theorem~\ref{thm:modular-magic-sq} for any finite Abelian group $\Gamma$ \cite{Cichacz-Hincz-2}.
\begin{theorem}[Cichacz, Hinc \cite{Cichacz-Hincz-2}]\label{thm:gamma-magic-rect}
	Let $\Gamma$ be an Abelian group of order $\Gamma=mn$. 	There exists a non-trivial $\Gamma$-magic rectangle MR$_{\Gamma}(m,n)$ if and only if $m>1,n>1$ and $m$ and $n$ are both even or $\Gamma\in\gr$.
\end{theorem}

The existence of $Z_{n^2}$-magic squares is a trivial consequence of the existence of magic squares in integers. We state the theorem here for completeness. For an overview, see, e.g., Chapter 6 Section 33 in~\cite{handbook}.

\begin{theorem}[\cite{handbook}] \label{thm:msquare}
	There exists a  {magic square} of side $n$ if and only if $n>2$.
\end{theorem}

Replacing every element $a$ in an integer magic square by $a-1$ or replacing only $n^2$ by 0 and performing addition modulo $n^2$ gives the result for $Z_{n^2}$-magic squares immediately. In some constructions, we will be referring to the existence of magic squares in integers. In that case, we will always consider the entries $0,1,\dots,n^2-1$. The theorem follows immediately.


\begin{theorem}\label{thm:cyclic}
	There exists a  $Z_{n^2}$-magic square of side $n$ for every \mmag{$n>2$}.
\end{theorem}

{
It is easy to check that the condition $n>2$ is not only sufficient but also necessary for both Abelian groups of order four, $Z_4$ and $Z_2\oplus Z_2$. All $Z_4$- and $Z_2\oplus Z_2$-squares (up to symmetry) are shown in Figure~\ref{fig: n=2}.

\begin{figure}[H]
	$$
		\begin{tabular}{|c|c|}
		\hline
		 0&1  \\
		\hline
		 3&2  \\
		\hline
		\end{tabular}
\hskip20pt
		\begin{tabular}{|c|c|}
		\hline
		 0&2  \\
		\hline
		 3&1  \\
		\hline
		\end{tabular}
\hskip20pt
		\begin{tabular}{|c|c|}
		\hline
		 (0,0)&(0,1)  \\
		\hline
		 (1,0)&(1,1)  \\
		\hline
		\end{tabular}
\hskip20pt
		\begin{tabular}{|c|c|}
		\hline
		 (0,0)&(1,1)  \\
		\hline
		 (1,0)&(0,1)  \\
		\hline
		\end{tabular}
\hskip20pt
		\begin{tabular}{|c|c|}
		\hline
		 (0,0)&(1,1)  \\
		\hline
		 (0,1)&(1,0)  \\
		\hline
		\end{tabular}
	$$
\caption{$Z_4$- and $Z_2\oplus Z_2$-squares of side 2}
\label{fig: n=2}	
\end{figure}

Therefore, we can state the following.

\begin{obs}\label{obs:necessary}
	There is no $\Gamma$-magic square $\gms(2)$.
\end{obs}
}

Sun and Yihui~\cite{Sun-Yihui} (without proper proof) made the following claim. 

\begin{claim}[Sun, Yihui~\cite{Sun-Yihui}]\label{clm:gamma-magic-sq}
	For every $n\geq3$ and any Abelian group $\Gamma$ of order $n^2$ there exists a $\Gamma$-magic square $\gms(n)$ of side $n$.
\end{claim}

They in fact only presented a construction for a $Z_n\oplus Z_n$-magic square $\ms{Z_n\oplus Z_n}(n)$ for $n$ odd.

\begin{theorem}[Sun, Yihui~\cite{Sun-Yihui}]\label{thm: Z_n x Z_n odd}
	For every odd $n\geq3$ there exists a  ${Z_n\oplus Z_n}$-magic square  $\ms{Z_n\oplus Z_n}(n)$ of side $n$.
\end{theorem}

%% file: 30_tools_22b.tex

\section{Tools: Latin squares and Kotzig arrays}\label{sec:LS and KA}

We will often utilize the Fundamental Theorem of Finite Abelian Groups and write $\Gamma=Z_{p_1^{s_1}}\oplus Z_{p_2^{s_2}}\oplus\dots\oplus Z_{p_t^{s_t}}$ where $p_i\leq p_{i+1}$ for $i=1,2,\dots,t-1$ are primes, $s_i$ are positive integers and if $p_i=p_{i+1}$, then $s_i\leq s_{i+1}$. However, we may occasionally change the order of primes when convenient.

Some of our constructions are based on \emph{mutually orthogonal Latin squares.} We provide the definitions slightly adjusted to our needs.

\begin{definition}\label{def:LS}
	A \emph{Latin square} LS$(n)$ of side $n$ is an $n\times n$ array with entries $a_{i,j}\in Z_n, i,j=1,2,\dots,n$, such that every row and every column contains each element of $Z_n$ exactly once. A Latin square is \emph{doubly diagonal} if moreover both the main forward and backward diagonals contain each element exactly once.
\end{definition}

	It is well known that Latin squares exist for every side $n>2$. For an overview, see, e.g., Chapter 3 in~\cite{handbook}.

\begin{theorem}[\cite{handbook}] \label{thm:LS}
	There exists a Latin square of side $n$ if and only if $n>2$.
\end{theorem}

	Another well known and extensively studied notion is \emph{mutually orthogonal Latin squares.}

\begin{definition}\label{def:MOLS}
	Two Latin squares of side $n$, $A_n=\{a_{i,j}\}_{i,j=1}^n$ and $B_n=\{b_{i,j}\}_{i,j=1}^n$, are \emph{mutually orthogonal} if the ordered pairs $(a_{i,j},b_{i,j})$ are all distinct. In the language of Abelian groups, they form the group $Z_n\oplus Z_n$.	
\end{definition}

The existence of pairs of mutually orthogonal Latin squares (MOLS) is well known.

\begin{theorem}[\cite{handbook}] \label{thm:MOLS}
	There exists a pair of mutually orthogonal Latin squares of side $n$ if and only if $n>2$ and $n\neq6$.
\end{theorem}

A corresponding result on doubly diagonal mutually orthogonal Latin squares (DDMOLS) by Heinrich and Hilton~\cite{Heinrich-Hilton} will be an important building block in our construction.

\begin{theorem}[Heinrich, Hilton~\cite{Heinrich-Hilton}] \label{thm:diag-MOLS}
	There exists a pair of doubly diagonal mutually orthogonal  Latin squares of side $n$ if and only if $n\neq2, 3, 6$ and possibly $10$. 
\end{theorem}

{The paper actually lists some unresolved cases (10, 12, 14, 15, 18, and 26), but a MathSciNet review MR0710888 states that ``Subsequent research has shown that all the above orders can be realized except possibly 10.'' We were unable to find where these cases were solved.}

In \cite{M-W} Marr and Wallis gave a definition of a Kotzig array as follows.

\begin{definition}\label{def: KA}
	A \emph{Kotzig array} $\ka(j,k)$ is
	a $j\times k$ grid where each row is a permutation of $\{1,2,\dots,k\}$ and each column has the same sum $j(k+1)/2$. 
\end{definition}

Note that a  Latin square  is a special case of a Kotzig array. Moreover a Kotzig array 
has the “magic-like” property with respect to the sums of rows and columns, but at the same time allows entry repetitions in columns. Kotzig arrays play an important role in graph labelings and will be used in our constructions.  

\begin{theorem}[Marr, Wallis \cite{M-W}]\label{Wal}
A Kotzig array of size $j \times k$ exists if and only if
$j>1$ and $j(k - 1)$ is even. 
\end{theorem}

  In~\cite{CicZ}, the first author introduced a generalization of Kotzig arrays and gave necessary and sufficient conditions for their existence. Namely, for an Abelian group $\Gamma$ of order $k$ we define a $\Gamma$-\textit{Kotzig array}.
  
\begin{definition}\label{def: Gamma-KA}
	A \emph{$\Gamma$-Kotzig array} $\gka(j,k)$ where $\Gamma$ is an Abelian group of order $k$  is
	a $j\times k$ grid where each row is a permutation of elements of $\Gamma$ and each column has the same sum. 
\end{definition}
  
  By adding to all elements in each row the inverse of its first entry we get all  elements in the first column equal  $0$ and therefore  all  column sums will equal to $0$.

  Recall that by $\mathcal{G}$ we denote the set of all finite Abelian groups that are either of odd order, or contain more than one involution.

\begin{theorem}[Cichacz \cite{CicZ}]\label{thm:Kotzig}
A $\Gamma$-Kotzig array of size $j \times k$ exists if and only if
$j>1$ and $j$ is even or $\Gamma\in \gr$. Moreover, whenever a $\Gamma$-Kotzig array exists, then there exists one with the column sums equal to $0$.
\end{theorem}

This theorem will be one of our main building stones for $\Gamma$-magis squares of even side.

%% file: 40_prelim_results_22g.tex

\section{Preliminary results}\label{sec:prelim}

In this section, we prove a series of observations and lemmas that will then allow us to construct $\Gamma$-magic squares $\gms(n)$ for all {$n>2$} and all Abelian groups $\Gamma$ of order $n^2$.

We start with a useful lemma.

\begin{lemma}\label{lem:SWL}
	Let $\Gamma$ be an Abelian group of order $n^2$. 
	Let  $\Gamma\cong \Gamma_0\oplus H$ for some groups $\Gamma_0$ with $|\Gamma_0|=m^2$, $m>1$ and $h$ with $|H|=k^2$. If there exists a $\Gamma_0$-magic square   MS$_{\Gamma_0}(m)$   with the  magic sum $\delta$ and {an $H$-Kotzig array of size $m\times k^2$}, then there exists a $\Gamma$-magic square MS$_{\Gamma}(n)$ with the magic sum $(k\delta,0)$.
\end{lemma}

\begin{proof}

Denote by $y_{i,j}$ the element in the $i$-th row and $j$-th column of the   MS$_{\Gamma_0}(m)$. Denote by $h_{i,s}$ the $i$-th row and $s$-th column of the $H$-Kotzig array.  Recall that all column sums are $\sum_{i=1}^mh_{i,s}=0$. Moreover, one can easily see that if $m$ is even, then there exists such an $H$-Kotzig array that  $h_{2j,s}=-h_{2j-1,s}=-h_{1,s}$ for $j=1,2,\ldots,m/2$.

Using the Kotzig array we build $k^2$ different $m\times m$ $H$-\textit{residual squares} $R_H^s(m)$ with entries $r^s_{i,j}\in H$
for $1 \leq s \leq k^2$.  In the first column of a given $R_H^s(m)$, we place the $s$-th column of $H$-Kotzig array. That is, $r^s_{i,1} =h_{i,s}$.

In the following $m-1$ columns, we place a circulant array constructed from the first column.   Namely, set 
$$r^s_{i,j}=r^s_{(i+j-1)(\mathrm{mod}\, m),1}, \;\; \text{for} \;\; j =2,3,\ldots,m.$$

Note that  $r^s_{i,j}\neq r^{s'}_{i,j}$ for any $s\neq s'$. Observe that  the sum of every  column is  $\sum_{j=1}^mr^s_{i,j}=\sum_{j=1}^mh_{j,s}=0$ and also the sum of each row is $\sum_{i=1}^mr^s_{i,j}=\sum_{j=1}^mh_{j,s}=0$.

\vskip10pt
\noindent
\textit{\underline{Case 1}} \ \ {$m$ is odd.}
\begin{adjustwidth}{25pt}{0pt}

We will now construct $k^2$ squares $Z^1,\ldots,Z^{k^2}$ of size $m\times m$. Denote by $z^s_{i,j}$ the entry in the $i$-th row and $j$-th column of the $s$-th square $X^s$. Recall that $\Gamma\cong \Gamma_0\oplus H$, therefore any $g\in \Gamma$ we can identify with a pair $(a,b)$ such that $a\in\Gamma_0$ and $b\in H$. 
Let
	$$z_{i,j}^s= (y_{i,j},r^s_{i,j})\,\,\mathrm{for}\,\,s=1,2,\ldots,k^2.$$
Note that the sum on every  column is
  $$
  	\sum_{j=1}^mz^s_{i,j}
  		=\left(\sum_{j=1}^my_{i,j},\sum_{j=1}^mr^s_{i,j}\right)
  		=(\delta,0),
  $$ 
  the sum of each row is 
  $$
  	\sum_{i=1}^mz^s_{i,j}
  		=\left(\sum_{i=1}^my_{i,j},\sum_{i=1}^mr^s_{i,j}\right)
  		=(\delta,0),
  $$  
  the diagonal 
  $$
	\sum_{i=1}^mz^s_{i,i}
		=\left(\sum_{i=1}^my_{i,i},\sum_{i=1}^mr^s_{i,i}\right)
		=\left(\delta,\sum_{j=1}^mh_{j,s}\right)=(\delta,0),
  $$ 
  and on the backward diagonal is $$\sum_{i=1}^mz^s_{m+1-i,i}
  	=\left(\sum_{i=1}^my_{m+1-i,i},\sum_{i=1}^mr^s_{m+1-i,i}\right)
  	=\left(\delta,m h_{m,s}\right).$$ 

Suppose that $z_{i,j}^s= z_{i',j'}^{s'}$.  Then $(y_{i,j},r^s_{i,j})=(y_{i',j'},r^{s'}_{i',j'})$,  but this is impossible by the definition of residual rectangles for $s\neq s'$.   Thus,  every element of $\Gamma$ appears exactly once in  $Z^1,\ldots,Z^{k^2}$.

We will construct now $k^2$ squares $X^1,\ldots,X^{k^2}$ of size $m\times m$. We will glue together those $k^2$ squares to obtain a MS$_\Gamma(n)$.  Denote by $x^s_{i,j}$ the entry in the $i$-th row and $j$-th column of the $s$-th square $X^s$.  Let
	$$x_{i,j}^s=
	\begin{cases}
		z_{i,j}^s&\mathrm{for}\;\;s=1,2,\ldots,k^2-k,\\
		z^s_{i,m+1-j}&\mathrm{for }\;\;s=k^2-k+1,k^2-k+2,\ldots,k^2\\
	\end{cases}$$
	All  column, row and backward diagonal sums in $X^s$ are equal to a constant $(\delta,0)$, and on the main diagonal to $(\delta,mh_{s,m})$ for $s=1,2,\ldots,k^2-k$, whereas for $s=k^2-k+1,k^2-k+2,\ldots,k^2$  backward diagonal sums equal to a constant
	$$
		\sum_{i=1}^mx^s_{m+1-i,i}=\sum_{i=1}^mz^s_{m+1-i,m+1-i}=(\delta,0),
	$$ 
	and the main diagonal sums are 
	$$
		\sum_{i=1}^mx^s_{i,i}=\sum_{i=1}^mz^s_{i,m+1-i}=(\delta,mh_{s,m}).
	$$ 
Moreover, without loss of generality we can assume that, if $k$ is odd, then $h_{k^2-(k-1)/2,m}=0$ (thus both diagonals in $X^{k^2-(k-1)/2}$ equal to $(\delta,0)$).

	We will glue now the rectangles $X^1,X^2,\ldots,X^{k^2}$ together in such a way that $X^{k^2-k+1},$ $X^{k^2-k+2},\ldots,X^{k^2}$ are on  the backward diagonal.

	Therefore, all  column, row, diagonal and backward diagonal sums of MS$_\Gamma(n)$ are equal to the constant $(k\delta,0)$, 
\end{adjustwidth}

\noindent
\textit{\underline{Case 2}} \ \  $m$ is even.

\begin{adjustwidth}{25pt}{0pt}

	Since we assumed that in this case the  $H$-Kotzig array has the property that $h_{2j,s}=-h_{2j-1,s}=-h_{1,s}$ for $j=1,2,\ldots,m$,  the sum on the main diagonal it is $\sum_{i=1}^mr^s_{i,i}=m  h_{1,s}$, whereas on the backward diagonal is $\sum_{i=1}^mr^s_{n+1-i,i}=m h_{m,s}=-m h_{1,s}$.

	We will construct now $k^2$ squares $X^1,\ldots,X^{k^2}$ of size $m\times m$. We will glue together those $k^2$ squares to obtain a MS$_\Gamma(n)$
	Denote by $x^s_{i,j}$ the entry in the $i$-th row and $j$-th column of the $s$-th square $X^s$. Recall that $\Gamma\cong \Gamma_0\oplus H$, therefore any $g\in \Gamma$ we can identify with a pair $(a,b)$ such that $a\in\Gamma_0$ and $b\in H$. 

	Let
	$$
		x_{i,j}^s=(y_{i,j},r^s_{i,j})\;\;\mathrm{for}\;\;s=1,2,\ldots,k^2.
	$$
	All  column and row sums in $X^s$ equal to a constant $(\delta,0)$,  the sum on the diagonal is $(\delta,m h_{1,s})$, on the backward diagonal is $(\delta,-m h_{1,s})$. Without loss of generality, we can assume that $h_{1,1}=0\in H$. 

	Let $X^s$ be such a square that $m  h_{1,s}\neq 0\in H$, we call such a square \emph{non-zero}. Note that if $X^s$ is a non-zero square,  then since $m$ is even there exists $s'\neq s$ such that $X^{s'}$ is non-zero and the sum on the diagonal of  $X^{s'}$ is $(\delta,m h_{1,s'})=(\delta,-m h_{1,s})$. We will call the rectangles $X^s$ and $X^{s'}$ \textit{complementary}. We will glue now the rectangles $X^1,X^2,\ldots,X^{k^2}$ in such a way that if $k$ is odd, then $X^1$ is in the center (i.e. this is this square which has elements on both the main forward and backward diagonal). Additionally, if there $X^s$ on main diagonal, its complementary square $X^{s'}$ is on the same diagonal.

The remaining part of the proof is similar to Case 1.
\end{adjustwidth}
\end{proof}

\begin{exm}\label{exm: Z_6 x Z_6}
	In Figure~\ref{fig: Z_6 x Z_6} we show the construction of $\ms{Z_{3}\oplus Z_3\oplus  Z_2\oplus  Z_2}(6)$, when a  $\ms{ Z_{3}\oplus Z_3}(3)$ and a $Z_{2}\oplus Z_2$-Kotzig array $K$ of size $3\times 4$ are given.
\end{exm}

\begin{figure}
\begin{subfigure}[t]{0.5\linewidth}
$$
\begin{array}{|c|c|c|}
\hline
	(0,0)& (0,1)& (0,2)\\\hline
	(1,0)& (1,1)& (1,2)\\\hline
	(2,0)& (2,1)& (2,2)\\\hline
\end{array}
$$
\caption{$\ms{Z_3\oplus Z_3}(3)$}
\end{subfigure}
\hfill
\begin{subfigure}[t]{0.5\linewidth}
$$
\begin{array}{|c|c|c|c|c|c|}
\hline
	(0,0)& (0,1) & (1,0) & (1,1)\\\hline
	(0,0)& (1,0) & (1,1) & (0,1)\\\hline
	(0,0)& (1,1) & (0,1) & (1,0)\\\hline
\end{array}
$$
\caption{$3\times 4$ Kotzig array}
\end{subfigure}
\begin{subfigure}[h]{1\textwidth}
$$
\begin{array}{|c|c|c|c|c|c|}
	\hline
	(0,0,0,0)& (0,1,0,0)& (0,2,0,0)& (0,2,0,1)& (0,1,1,1)& (0,0,1,0)\\\hline
	(1,0,0,0)& (1,1,0,0)& (1,2,0,0)& (1,2,1,0)& (1,1,0,1)& (1,0,1,1)\\\hline
	(2,0,0,0)& (2,1,0,0)& (2,2,0,0)& (2,2,1,1)& (2,1,1,0)& (2,0,0,1)\\\hline
	(0,2,1,0)& (0,1,0,1)& (0,0,1,1)& (0,0,0,1)& (0,1,1,0)& (0,2,1,1)\\\hline
	(1,2,1,1)& (1,1,1,0)& (1,0,0,1)& (1,0,1,0)& (1,1,1,1)& (1,2,0,1)\\\hline
	(2,2,0,1)& (2,1,1,1)& (2,0,1,0)& (2,0,1,1)& (2,1,0,1)& (2,2,1,0)\\\hline
\end{array}
$$
\end{subfigure}
\caption{$\ms{ Z_{3}\oplus Z_3\oplus  Z_2\oplus  Z_2}(6)$}
\label{fig: Z_6 x Z_6}
\end{figure}

The following observation will be needed as a starting case in a later lemma.

\begin{obs}\label{obs:general H+A}
	Let $\Gamma=H\oplus K$  where $|H|=k^2m^2$ and  $|K|=k^2$. If there exists an $\ms{H}(km)$, then there exists an $\gms(k^2m)$.
\end{obs}

\begin{proof}
When $k>1$ is odd, there exists a {$kn\times k^2$}  $K$-Kotzig array for any $n$. When $k$ is even, then {$kn$}  is even and there again exists a {$kn\times k^2$}  $K$-Kotzig array. 
Therefore, there exists a $K$-Kotzig array of size $km\times |K|$ by Theorem~\ref{thm:Kotzig}.
Hence, we are done by  Lemma~\ref{lem:SWL}.
\end{proof}


When $\Gamma$ has no subgroups $Z_{p^{2\alpha}}$, we need a different reduction technique. Our Lemma~\ref{lem: Z_p^c + Z_p^c+d} below provides one. We start with some simple results that will be used as building blocks in the lemma.

\begin{obs}\label{obs:base-non-cyclic}
	Let $n\geq3$ and $\Gamma=Z_n\oplus Z_n$. Then there exists a $\Gamma$-magic square $\gms(n)$ of side $n$.
\end{obs}

\begin{proof}
	For $n\neq3, 6, 10$, the proof follows from Theorem~\ref{thm:diag-MOLS} on the existence of DDMOLS of side $n$. An example for $n=3$ is given in Figure~\ref{fig: Z_6 x Z_6}. 
   
	Because for $n=3,5$ there exists a $\ms{Z_n\oplus Z_n}(n)$

	{by Theorem~\ref{thm: Z_n x Z_n odd}} and there also exists a $Z_2\oplus Z_2$-Kotzig array of size $n\times 4$   {by Theorem~\ref{thm:Kotzig}}, we can apply Lemma~\ref{lem:SWL} to prove the existence of $\gms(6)$ and $\gms(10)$ because  $Z_{2n}\oplus Z_{2n}\cong Z_n\oplus Z_n\oplus Z_2\oplus Z_2$.
\end{proof}
 

 \begin{obs}\label{obs: Z_2 + Z_2^c}
 	Let $\Gamma= Z_{2}\oplus Z_{2^{2\alpha-1}}$ and $\alpha\geq2$. Then there exists an $\gms(2^\alpha)$.
 \end{obs}

 \begin{proof} Denote by $a_{i,j}$ the $i$-th row and $j$-th column of the   $\gms(2^\alpha)$. For computational convenience, we number the rows and columns  $0,1,\dots,2^\alpha -1$.
 We provide direct construction. We set

 	\begin{align*}
 		a_{0,j} &=
 	\begin{cases}
 	 	(0,j) \ \ \ \ \ \: \,  \text{for} \ j=0,1,\dots,2^{\alpha-1}-1\\
 	 	(0,j+1)  \: \,  \text{for} \ j=2^{\alpha-1},2^{\alpha-1}+2,\dots,2^{\alpha}-2,\\
 	 	(0,j-1) \ \, \text{for} \ j=2^{\alpha-1}+1,2^{\alpha-1}+3,\dots,2^{\alpha}-1,
 	\end{cases}
 	\\
 		a_{1,j} &= (0,2^{\alpha+1}-1) - a_{0,j},\\
 		a_{i,j} &= (0,2^{\alpha+1}) + a_{i-2,j}  \ \text{for} \ i=2,3,\dots\,2^{\alpha-1}-1, \ \text{and}\\
 		a_{i,j} &= a_{i-2^{\alpha-1},j} + (1,0)  \ \: \text{for} \ i=2^{\alpha-1},2^{\alpha-1}+1,\dots,2^{\alpha}-1.\\
 	\end{align*}

	Later we will need some of the above in terms of explicit formulas. In particular, for every even $i=0,2,\dots,2^{\alpha}-2$ we have
	$$
		a_{i,j} + a_{i+1,j} = a_{i,j} + ((0,2^{\alpha+1}-1) -a_{i,j} ) 
		= (0,2^{\alpha+1}-1).
	$$
	Similarly, for $i=0,2,\dots,2^{\alpha-1}-2$, 
	\begin{align*}
		a_{i,i} + a_{i+1,i+1} 
		&= a_{i,i} + ((0,2^{\alpha+1}-1) -a_{i,i+1} )\\ 
		&= {(0,i) + (0,2^{\alpha+1}-1) - (0,i+1)}\\
		&= (0,2^{\alpha+1}-2),
	\end{align*}
	and for $i=2^{\alpha-1},2^{\alpha-1}+2,\dots,2^{\alpha}-2$, 

	\begin{align*}
		a_{i,i} + a_{i+1,i+1} 
		&= a_{i,i} + ((0,2^{\alpha+1}-1) -a_{i,i+1} )\\ 
		&= (1,i+1) + (0,2^{\alpha+1}-1) - (1,(i+1)-1)\\
		&= (0,2^{\alpha+1}).
	\end{align*}
	Hence, for $i=0,2,\dots,2^{\alpha-1}-2$ we obtain
	\begin{align*}
		a_{i,i} + a_{i+1,i+1} + a_{2^{\alpha-1}+i,2^{\alpha-1}+i} + a_{2^{\alpha-1}+i+1,2^{\alpha-1}+i+1}
		&= (0,2^{\alpha+1}-2) + (0,2^{\alpha+1})\\
		&=(0,2^{\alpha+2}-2).		
	\end{align*}
	Finally, for $i=0,2,\dots,2^{\alpha-1}-2$, 
	\begin{align*}
		a_{i,2^{\alpha}-i-1} + a_{i+1,2^{\alpha}-i-2} 
		&= a_{i,2^{\alpha}-i-1}+((0,2^{\alpha+1}-1)-a_{i,2^{\alpha}-i-2})\\ 
		&= (0,2^{\alpha}-i-2) + (0,2^{\alpha+1}-1) - (0,2^{\alpha}-i-1)\\
		&= (0,2^{\alpha+1}-2),
	\end{align*}	
	and for $i=2^{\alpha-1},2^{\alpha-1}+2,\dots,2^{\alpha}-2$, 

	\begin{align*}
		a_{i,2^{\alpha}-i-1} + a_{i+1,2^{\alpha}-i-2} 
		&= a_{i,2^{\alpha}-i-1}+((0,2^{\alpha+1}-1)-a_{i,2^{\alpha}-i-2})\\ 
		&= (0,2^{\alpha}-i-1) + (0,2^{\alpha+1}-1) - (0,2^{\alpha}-i-2)\\
		&= (0,2^{\alpha+1}).
	\end{align*}
	Hence, for $i=0,2,\dots,2^{\alpha-1}-2$ we obtain
	\begin{align*}
		  a_{i,2^{\alpha}-i-1} 
		+ a_{i+1,2^{\alpha}-i-2} 
		+ a_{2^{\alpha}+i,2^{\alpha}-i-1} 
		+ a_{2^{\alpha}+i+1,2^{\alpha}-i-2} 
		&= (0,2^{\alpha+1}-2) + (0,2^{\alpha+1})\\
		&=(0,2^{\alpha+2}-2).		
	\end{align*}
 
	We will denote by 
	$\rho_i=(\rho_i^{(1)},\rho_i^{(2)})$ the sum of all entries in row $i$, 
	$\sigma_j=(\sigma_j^{(1)},\sigma_j^{(2)})$ the sum in column $j$, 
	$\delta=(\delta^{(1)},\delta^{(2)})$ the sum in the main forward diagonal, and 
	$\gamma=(\gamma^{(1)},\gamma^{(2)})$ the sum in the main backward  diagonal. We want to show that the row, column, and diagonal sums are all equal to the same magic constant $\mu=(\mu^{(1)},\mu^{(2)})$. To simplify our calculations, we observe that because $\alpha\geq2$, every row, column, and diagonal contains an even number of elements with the first entry equal to 1. Therefore, we have 
	$\rho^{(1)}_i=\sigma^{(1)}_j=\delta^{(1)}=\gamma^{(1)}=0$ for every $i,j$ and we only need to check the values of
	$\rho^{(2)}_i,\sigma^{(2)}_j,\delta^{(2)},\gamma^{(2)}$.

	The zero row contains elements $(0,0),(0,1),\dots,(0,2^\alpha -1)$, therefore we have
	\begin{align*}
 		\rho^{(2)}_0
 		&=\sum_{j=0}^{2^\alpha -1} j \\
 		&= (2^\alpha -1)2^\alpha/2 \\
 		&= (2^\alpha -1)2^{\alpha-1} \\
 		&=2^{2\alpha -1} -2^{\alpha-1} \\
 		&=-2^{\alpha-1}
	\end{align*}
	because the addition is performed in $Z_{2^{2\alpha-1}}$,
	and

	\begin{align*}
 		\rho^{(2)}_1
 		&=\sum_{j=0}^{2^\alpha -1}  (2^{\alpha+1}-1) - j \\
 		&= 2^{\alpha}(2^{\alpha+1}-1) - \rho^{(2)}_0\\
 		&= 2^{2\alpha+1}-2^{\alpha} + 2^{\alpha-1}\\
 		&= -2^{\alpha-1}(2 -1)\\
 		&= -2^{\alpha-1}
	\end{align*}
	as well.
	Now, because we have $a_{i,j} = (0,2^{\alpha+1}) + a_{i-2,j}$ for $i=2,3,\dots\,2^{\alpha-1}-1$,
	\begin{align*}
 		\rho^{(2)}_i
 		&=2^{\alpha}2^{\alpha+1} + \rho^{(2)}_{i-2} \\
 		&=\rho^{(2)}_{i-2}.
 	\end{align*}
	Thus, $\rho^{(2)}_i=-2^{\alpha-1}$ for $i=0,1,\dots,2^{\alpha}-1$.
	
	Now, because 
	$a_{i,j} = a_{i-2^{\alpha-1},j} + (1,0)$ for $i=2^{\alpha-1},2^{\alpha-1}+1,\dots,2^{\alpha}-1$, it should be obvious that $\rho^{(2)}_i=-2^{\alpha-1}$ and $\rho_i=(0,-2^{\alpha-1})$ for all $i$.
	
	For the columns sum, we recall that for every even $i$ and every $j$, we have
	$$
		a_{i,j} + a_{i+1,j} = (0,2^{\alpha+1}-1).
	$$
	Therefore,
	\begin{align*}
		\sigma_j
		&=\sum_{s=0}^{2^{\alpha-1}-1} a_{2s,j} + a_{2s+1,j}\\
		&=\sum_{s=0}^{2^{\alpha-1}-1} (0,2^{\alpha+1}-1)\\
		&= 2^{\alpha-1} (0,2^{\alpha+1}-1) \\
		&= (0,2^{2\alpha}-2^{\alpha-1}) \\
		&= (0,-2^{\alpha-1})
	\end{align*}
	because the calculation is performed in $Z_2\oplus Z_{2^{2\alpha-1}}$.

	Remembering that for even $i=0,2,\dots,2^{\alpha-1}-2$ we proved that 	
	\begin{align*}
		a_{i,i} + a_{i+1,i+1} + a_{2^{\alpha-1}+i,2^{\alpha-1}+i} + a_{2^{\alpha-1}+i+1,2^{\alpha-1}+i+1}
		&=(0,2^{\alpha+2}-2),		
	\end{align*}
	we have
	\begin{align*}
		\delta
		&= \sum_{s=0}^{2^{\alpha-2}-1}
		a_{2s,2s} + a_{2s+1,2s+1} 
			+ a_{2^{\alpha-1}+2s,2^{\alpha-1}+2s} 
			+ a_{2^{\alpha-1}+2s+1,2^{\alpha-1}+2s+1}\\
		&= \sum_{s=0}^{2^{\alpha-2}-1} (0,2^{\alpha+2}-2)\\
		&=2^{\alpha-2} (0,2^{\alpha+2}-2)\\
		&= (0,2^{2\alpha}-2^{\alpha-1})\\
		&= (0,-2^{\alpha-1}).
	\end{align*}
	Similarly, using the previously proved identity
 	$$
	  a_{i,2^{\alpha}-i-1} 
			+ a_{i+1,2^{\alpha}-i-2} 
			+ a_{2^{\alpha}+i,2^{\alpha}-i-1} 
			+ a_{2^{\alpha}+i+1,2^{\alpha}-i-2} 
			=(0,2^{\alpha+2}-2)
 	$$
	we obtain
	\begin{align*}
		\eta
		&= \sum_{s=0}^{2^{\alpha-2}-1}
	  a_{i,2^{\alpha}-i-1} 
			+ a_{i+1,2^{\alpha}-i-2} 
			+ a_{2^{\alpha}+i,2^{\alpha}-i-1} 
			+ a_{2^{\alpha}+i+1,2^{\alpha}-i-2} \\
		&= \sum_{s=0}^{2^{\alpha-2}-1} (0,2^{\alpha+2}-2)\\
		&= (0,-2^{\alpha-1}).
	\end{align*}
	Therefore, we have shown that	$\rho_i=\sigma_j=\delta=\gamma=(0,-2^{\alpha-1})$ for every $i,j$. This means that the square $\gms(2^\alpha)$ is $\Gamma$-magic with the magic constant $\mu=(0,-2^{\alpha-1})$, which completes the proof.
 \end{proof}

 	An example of $\gms(8)$ for $\Gamma=Z_{2}\oplus Z_{32}$, that is, $\alpha=3$, is given in Figure~\ref{fig:example Z_2 x Z_32}.

 \begin{figure}[H]
 \begin{center}
 	\begin{tabular}{|c|c|c|c|c|c|c|c|}
 		\hline
 		(0,0) &(0,1)  &(0,2)  &(0,3)  &(0,5)  &(0,4)  &(0,7)  &(0,6)\\ \hline
 		(0,15) &(0,14) &(0,13) &(0,12) &(0,10) &(0,11) &(0,8)  &(0,9)\\ \hline
 		(0,16) &(0,17) &(0,18) &(0,19) &(0,21) &(0,20) &(0,23) &(0,22)\\ \hline
 		(0,31) &(0,30) &(0,29) &(0,28) &(0,26) &(0,27) &(0,24) &(0,25)\\ \hline
 		(1,0) &(1,1)  &(1,2)  &(1,3)  &(1,5)  &(1,4)  &(1,7)  &(1,6)\\ \hline
 		(1,15) &(1,14) &(1,13) &(1,12) &(1,10) &(1,11) &(1,8)  &(1,9)\\ \hline
 		(1,16) &(1,17) &(1,18) &(1,19) &(1,21) &(1,20) &(1,23) &(1,22)\\ \hline
 		(1,31) &(1,30) &(1,29) &(1,28) &(1,26) &(1,27) &(1,24) &(1,25)\\ \hline
 	\end{tabular}
 \end{center}
 	\caption{$\ms{Z_{2}\oplus Z_{32}}(8)$}	
 	\label{fig:example Z_2 x Z_32}
 \end{figure}

 
 \vskip10pt 
 
 \begin{figure}[H]
 \begin{center}
 \begin{tabular}{|c|c|c|c|}
 \hline
     (0,0)  & (0,1) & (0,3) & (0,2)\\ \hline
     (0,7)  & (0,6) & (0,4) & (0,5)\\ \hline
     (1,0)  & (1,1) & (1,3) & (1,2)\\ \hline
     (1,7)  & (1,6) & (1,4) & (1,5)\\ \hline
   \end{tabular}
 \caption{$\ms{Z_2\oplus Z_8}(4)$}
 \label{fig: Z_2 x Z_8}
 \end{center}
 \end{figure}

  \begin{lemma}\label{lem:Z_2^2d+1 x Z_2^2d+3}
 	Let $\delta\geq0$ and\/ $\Gamma=Z_{2^{2\delta+1}}\oplus Z_{2^{2\delta+3}}$. Then there exists a $\Gamma$-magic square $\gms({2^{2\delta+2}})$.
 \end{lemma}

 \begin{proof}
 	For $\delta=0$, an example is shown in Figure~\ref{fig: Z_2 x Z_8}. For $\delta=1$, we present a direct construction. We first modify the $Z_2\oplus Z_{32}$-magic square shown in Figure~\ref{fig:example Z_2 x Z_32} by replacing $Z_2$ with the subgroup $\langle4\rangle$ of $Z_8$. We call this {residual} square 
 	$A=\{a_{i,j}\}_{i,j=1}^{16}$ and observe that is has the magic constant $\mu_A=(0,30)$. We denote the copies by $A_{s,t}=\{a_{i,j}^{s,t}\}, \ s,t\in\{1,2\}$ and place them in a $16\times16$-array $B$. Then we modify $B$ by adding the element 
 	$(0,0)\in Z_8\oplus Z_{32}$ to every entry in $A_{1,1}$, {$(1,0)$} to every entry in $A_{1,2}$, 
{$(2,0)$} to every entry in $A_{2,1}$, and  {$(3,0)$} to every entry in $A_{2,2}$.  We call these modified subsquares $B_{s,t}$.
 	
 	Now we check the partial magic constants in each subsquare $B_{s,t}$. Let $(g,0)\in Z_8\oplus Z_{32}$ be the element added to all entries of $A_{s,t}$. Then every row, column, and diagonal sum is equal to
 	$$
 		\mu_{B_{s,t}} =\mu_A + 8(g,0) =\mu_A + (8g,0)=\mu_A + (0,0)=\mu_A,
 	$$
 	because in $Z_8$ we have $8g=0$ for every $g$. It should be obvious now that $\mu_B=2\mu_A=(0,28)$.
 	
 	For $\delta\geq2$, we use Kronecker-type construction, using 

 	$\ms{Z_{2}\oplus Z_{2^{2\delta+3}}}(2^{\delta+2})$ 
 	and $\ms{}(2^{\delta})$ (with entries $0,1,\dots,2^{2\delta}-1)$.

  	Let $A=\{a_{i,j}\}_{i,j=1}^{2^\delta}$ be an 
  	$\ms{}(2^{\delta})$ (in integers) with magic constant $\mu_A$ and $B=\{b_{s,t}\}_{s,t=1}^{2^{\delta+2}}$ a 
  	${Z_{2}\oplus Z_{2^{2\delta+3}}}$-magic square with  entries 
 	$b_{s,t}=(b^{(1)}_{s,t},b^{(2)}_{s,t})$ for $1\leq s,t \leq 2^{\delta+2}$ and   magic constant $\mu_B=(\mu^{(1)}_{B},\mu^{(2)}_{B})$. Recall that the subgroup
 	$\langle 2^{2\delta}\rangle$ of $Z_{2^{2\delta+1}}$ is isomorphic to $Z_2$.
 	
 	We create $2^{2\delta}$ squares $B^{i,j}$ isomorphic to $B$ for $i,j=1,2,\dots,2^\delta$ defined as
 	$$
 		b^{i,j}_{s,t}=(2^{2\delta}b^{(1)}_{s,t},b^{(2)}_{s,t}), \ 
 		i,j=1,2,\dots,2^\delta, \ s,t=1,2,\dots2^{\delta+2}.
 	$$
 	and place them in an $2^{2\delta+2}\times 2^{2\delta+2}$ array $C'$ in a natural way.
 	One can observe that each such square $B^{i,j}$ is a 
 	$\langle {2^{2\delta}}\rangle \oplus  Z_{2^{2\delta+3}}$-magic square.
 
 	Now we create a $2^{2\delta+2}\times2^{2\delta+2}$ square $C$ by adding entry 
 	$(a_{i,j},0)$ to every element $b^{i,j}_{s,t}, \ s,t=1,2,\dots,2^{\delta+2}$. That is,
 	$$
 	c^{i,j}_{s,t}=(2^{2\delta}b^{(1)}_{s,t} + a_{i,j},b^{(2)}_{s,t}) \ 
 		i,j=1,2,\dots,2^\delta, \ s,t=1,2,\dots2^{\delta+2}.	
 	$$
 	Each $2^{\delta+2}\times2^{\delta+2}$ subsquare $C^{i,j}$ of $C$ now has each row, column, and both diagonal sums equal to
 	$$
 		\mu_{C^{i,j}}=(2^{2\delta}\mu^{(1)}_{B} + 2^{\delta}a_{i,j},\mu^{(2)}_{B}).
 	$$
 	Therefore, the row sum $\rho_u$ where $u=2^{\delta}(i-1)+s$ in $C$ is 
 	\begin{align*}
 	\rho_u 	&= \sum_{j=1}^{2^{\delta}}\mu_{C^{i,j}}\\
 			&= \sum_{j=1}^{2^{\delta}}
 				(2^{2\delta}\mu^{(1)}_{B} + 2^{\delta}a_{i,j},\mu^{(2)}_{B})\\
 			&= \sum_{j=1}^{2^{\delta}}
 				(2^{2\delta}\mu^{(1)}_{B},\mu^{(2)}_{B})+(2^{\delta}a_{i,j},0)\\
 			&= 2^{\delta}(2^{2\delta}\mu^{(1)}_{B},\mu^{(2)}_{B})
 			 + 2^{\delta}\sum_{j=1}^{2^{\delta}}(a_{i,j},0)\\
 			&= 2^{\delta}(2^{2\delta}\mu^{(1)}_{B},\mu^{(2)}_{B})
 			 + 2^{\delta}(\mu_A,0)\\
 			&= 2^{\delta}(2^{2\delta}\mu^{(1)}_{B}+\mu_A,\mu^{(2)}_{B}).
 	\end{align*}
 	The other constants for columns and diagonals can be checked essentially the same way. Therefore, the lemma holds.	
 \end{proof}


\begin{lemma}\label{lem:Z_k^2m x Z_m}
	Let $k,m\geq3$ and $\Gamma=Z_{k^2m}\oplus Z_m$. Then there exists a $\Gamma$-magic square $\gms(km)$.
\end{lemma}

\begin{proof}
%
%

	For convenience, we will number the rows and columns from 0 to $km-1$.
	It follows from Theorem~\ref{thm:msquare} and Observation~\ref{obs:base-non-cyclic} that there exists a 
	magic square $A=\{a_{i,j}\}_{i,j=0}^{k-1}$ (in integers) with magic constant $\mu_A$ and a $Z_m\oplus Z_m$-magic square $B$ with  entries 
	$b_{s,t}=(b^{(1)}_{s,t},b^{(2)}_{s,t})$ for $0\leq s,t \leq m-1$ and   magic constant $\mu_B=(\mu^{(1)}_{B},\mu^{(2)}_{B})$. 
	
	We use a Kronecker-type construction to obtain a $\Gamma$-magic square $C$.
	
	We express the coordinates in $C$ as 
		$u = q_1 k + r_1$ and  $v = q_2 k + r_2$ where $0\leq q_1,q_2\leq m-1$ and $0\leq r_1,r_2\leq k-1$. Then we define $c_{u,v}$ as 
		$$
			c_{u,v} = (k^2 b^{(1)}_{q_1,q_2} + a_{r_1,r_2},b^{(2)}_{q_1,q_2}).
		$$

	We now want to check the row, column, and diagonal sums. For row $u$, we have
	\begin{align*}
	\rho_u 	&= \sum_{v=0}^{km-1}c_{u,v}\\
			&= \sum_{v=0}^{km-1}(k^2 b^{(1)}_{q_1,q_2} 
				+a_{r_1,r_2},b^{(2)}_{q_1,q_2})\\
			&= \sum_{q_2=0}^{m-1}\sum_{r_2=0}^{k-1}(k^2 b^{(1)}_{q_1,q_2} 
							+a_{r_1,r_2},b^{(2)}_{q_1,q_2})\\
			&= \sum_{q_2=0}^{m-1}(k^3 b^{(1)}_{q_1,q_2} 
							+\mu_A,kb^{(2)}_{q_1,q_2})\\
			&= (k^3\mu^{(1)}_B + m\mu_A,k\mu^{(2)}_B).
	\end{align*}
	Similarly, for column $v$ we have
	\begin{align*}
	\sigma_v&= \sum_{u=0}^{km-1}c_{u,v}\\
			&= \sum_{q_1=0}^{m-1}\sum_{r_1=0}^{k-1}(k^2 b^{(1)}_{q_1,q_2} 
							+a_{r_1,r_2},b^{(2)}_{q_1,q_2})\\
			&= \sum_{q_1=0}^{m-1}(k^3 b^{(1)}_{q_1,q_2} 
							+\mu_A,kb^{(2)}_{q_1,q_2})\\
			&= (k^3\mu^{(1)}_B + m\mu_A,k\mu^{(2)}_B).
	\end{align*}
	For the main diagonal, 
	\begin{align*}
	\delta	&= \sum_{u=0}^{km-1}c_{u,u}\\
			&= \sum_{q_1=0}^{m-1}\sum_{r_1=0}^{k-1}(k^2 b^{(1)}_{q_1,q_1} 
							+a_{r_1,r_1},b^{(2)}_{q_1,q_1})\\
			&= \sum_{q_1=0}^{m-1}(k^3 b^{(1)}_{q_1,q_2} 
							+\mu_A,kb^{(2)}_{q_1,q_2})\\
			&= (k^3\mu^{(1)}_B + m\mu_A,k\mu^{(2)}_B).
	\end{align*}
	and for the backward diagonal we have
	\begin{align*}
	\gamma	&= \sum_{u=0}^{km-1}c_{u,km-1-u}\\
			&= \sum_{q_1=0}^{m-1}\sum_{r_1=0}^{k-1}(k^2 b^{(1)}_{q_1,m-1-q_1}
							+a_{r_1,k-1-r_1},b^{(2)}_{q_1,m-1-q_1})\\
			&= \sum_{q_1=0}^{m-1}(k^3 b^{(1)}_{q_1,m-1-q_1}
							+\mu_A,kb^{(2)}_{q_1,m-1-q_1})\\
			&= (k^3\mu^{(1)}_B + m\mu_A,k\mu^{(2)}_B).
	\end{align*}
	Because all the sums are equal, we have shown that $C$ is a $\Gamma$-magic square with magic constant $\mu=(k^3\mu^{(1)}_B + m\mu_A,k\mu^{(2)}_B)$.	
\end{proof}


\begin{exm}\label{exm: Z_27 x Z_3}
	In Figure~\ref{fig: Z_27 x Z_3} we  show the construction of $\ms{Z_{27}\oplus Z_3}(9)$.
\end{exm}

\begin{figure}[h]
\begin{subfigure}[t]{0.5\linewidth}
$$
A=\begin{array}{|c|c|c|}
\hline
	7& 0& 5\\\hline
	2& 4& 6\\\hline
	3& 8& 1\\\hline
\end{array}
$$
\caption{$\ms{Z_9}(3)$}
\end{subfigure}
\hfill
\begin{subfigure}[t]{0.5\linewidth}
$$
B=\begin{array}{|c|c|c|}
\hline
	(0,0)& (0,1)& (0,2)\\\hline
	(1,0)& (1,1)& (1,2)\\\hline
	(2,0)& (2,1)& (2,2)\\\hline
\end{array}
$$
\caption{$\ms{Z_3\oplus Z_3}(3)$}
\end{subfigure}
\\
\begin{subfigure}[h]{1\textwidth}
$$
\begin{array}{|c|c|c|c|c|c|c|c|c|}
	\hline
	(7,0)& (0,0)& (5,0)& (7,1)& (0,1)& (5,1)&(7,2)& (0,2)& (5,2)\\\hline
	(2,0)& (4,0)& (6,0)& (2,1)& (4,1)& (6,1)&(2,2)& (4,2)& (6,2)\\\hline
	(3,0)& (8,0)& (1,0)& (3,1)& (8,1)& (1,1)&(3,2)& (8,2)& (1,2)\\\hline

	(16,0)& (9,0)& (14,0)& (16,1)& (9,1)& (14,1)&(16,2)& (9,2)& (14,2)\\\hline
	(11,0)& (13,0)& (15,0)& (11,1)& (13,1)& (15,1)&(11,2)& (13,2)& (15,2)\\\hline
	(12,0)& (17,0)& (10,0)& (12,1)& (17,1)& (10,1)&(12,2)& (17,2)& (10,2)\\\hline

	(25,0)& (18,0)& (23,0)& (25,1)& (18,1)& (23,1)&(25,2)& (18,2)& (23,2)\\\hline
	(20,0)& (22,0)& (24,0)& (20,1)& (22,1)& (24,1)&(20,2)& (22,2)& (24,2)\\\hline
	(21,0)& (26,0)& (19,0)& (21,1)& (26,1)& (19,1)&(21,2)& (26,2)& (19,2)\\\hline
\end{array}
$$
\end{subfigure}
\caption{$\ms{ Z_{27}\oplus Z_3(9)}$}
\label{fig: Z_27 x Z_3}
\end{figure}


\begin{obs}\label{obs:Z_2gamma+Z_4}
	Let $\Gamma=Z_{4}\oplus Z_{2^{2\gamma}}$  where $\gamma\geq2$. Then there exists a $\gms(2^{\gamma+1})$.
\end{obs}

\begin{proof}
	Let 
	$A(2^\gamma)=\{a_{i,j}\}^{2^\gamma-1}_{i,j=0}=\ms{Z_{2^{2\gamma}}}(2^{\gamma})$, which exists by Theorem~\ref{thm:cyclic}. Construct {a residual square} $B(2^{\gamma+1})=\{b_{i,j}\}^{2^\gamma}_{i,j=0}=\gms(2^{\gamma+1})$ as follows. First place four copies of $A(2^{\gamma})$ into a square $B'(2^{\gamma+1})$ in the natural way, that is, each $A(2^{\gamma})$ fills one corner of $B'(2^{\gamma+1})$. We denote the copies by $A_{s,t}$ where $s,t\in\{1,2\}$.

	For  $\Gamma=Z_4\oplus Z_{2^{2\gamma}}$ we obtain $\gms(2^{\gamma+1})=B(2^{\gamma+1})$ by adding an element of $Z_4$ as the first entry to each label. 
	That is, we replace each $a_{i,j}$ in a given $A_{s,t}$ by $(b,a_{i,j})$, $b=2s+t-3\in Z_4$, as follows.	
	Simply said, to labels 
	in $A_{1,1}$ we add 0, 
	in $A_{1,2}$ we add 1,
	in $A_{2,1}$ we add 2, and
	in $A_{2,2}$ we add 3.

	Now, when an element of $Z_4$ appears in a row, column, or one of the two diagonals, it always appears there exactly $2^\gamma$ times, precisely in one of the squares $A_{s,t}(2^\gamma)$. Because $\gamma\geq2$, we see that 4 divides $2^\gamma$ and it immediately follows that the new entry of the magic constant will be always 0. 
\end{proof}


\begin{lemma}\label{lem: Z_p^c + Z_p^c+d}

	Let $\Gamma= Z_{p^{\beta}}\oplus Z_{p^{\gamma}}$ where $1\leq\beta\leq\gamma$, $\beta+\gamma=2\alpha$, 
	and $|\Gamma|=p^{2\alpha}>4$. Then there exists an $\gms(p^\alpha)$.
\end{lemma}

\begin{proof}
	We treat the cases of $\beta$ even and odd separately.

\vskip6pt\noindent
\textit{\underline{Case 1} \ $\beta$ is even}

\begin{adjustwidth}{40pt}{0pt}
	We set $\beta=2\delta$; then $\gamma$ is also even, say $\gamma=2\lambda$. 
	When $p=2$ and $\beta=2$, then there exists $\gms(2^{\lambda+1})$ by Observation~\ref{obs:Z_2gamma+Z_4}.

	When either $p=2$ and $\beta>2$, or $p$ is odd, then by Theorem~\ref{thm:cyclic} there exist $\ms{Z_{p^{\beta}}}(p^{\delta})$ and a $Z_{p^{\gamma}}$-Kotzig array of size $p^{\delta}\times p^{\gamma}$ by Theorem~\ref{thm:Kotzig}. Hence, we are done by  Lemma~\ref{lem:SWL}.
\end{adjustwidth}

\vskip6pt\noindent
\textit{\underline{Case 2} \ $\beta$ is odd}
	
\begin{adjustwidth}{40pt}{0pt}

	We set $\beta=2\delta+1$; then $\gamma$ is also odd, say $\gamma=2\lambda+1$ and proceed by induction on $\beta$.
When $p=2$ and $\beta=1$, we use Observation~\ref{obs: Z_2 + Z_2^c}.
When $p>2$ and  $\beta=1$, we apply Lemma~\ref{lem:Z_k^2m x Z_m} for $m=p,k=p^{\lambda}$. Finally assume  $\beta\geq3$, then we set $\gamma=\beta+2\kappa$,  $\kappa\geq1$. If now $k=p^{\kappa}>2$, then apply 	Lemma~\ref{lem:Z_k^2m x Z_m} for $m=p^{\beta}$, $k=p^{\kappa}$, otherwise we can use Lemma~\ref{lem:Z_2^2d+1 x Z_2^2d+3} since $\Gamma=Z_{2^{2\delta+1}}\oplus Z_{2^{2\delta+3}}$.
\end{adjustwidth}
\vskip-17pt
\end{proof}


%% file: 50_construction_n_odd_22a.tex

\section{Main construction}\label{sec: const}

Now when we have our toolbox complete, we proceed to constructions. We begin with the simplest case for squares with an odd side.


\subsection{Construction for $n$ odd}\label{sec: n odd}

The tools from Section~\ref{sec:prelim} now allow immediate proof for all groups $\Gamma$ of an odd order.

\begin{theorem}\label{thm:main-odd}
	Let $\Gamma$ be an Abelian group of order $n^2$ where $n\geq3$ is odd. Then there exists a $\Gamma$-magic square  $\gms(n)$ of side $n$.
\end{theorem}

\begin{proof}
	By Fundamental Theorem of Finite Abelian Groups the group $\Gamma$ is either  cyclic and the result follows directly from Theorem~\ref{thm:cyclic}, or $\Gamma\cong H\oplus K$, where $|H|=p^{2\lambda}$ for some prime number $p>2$ and integer $\lambda\ge 1$. In the latter case, we either have $H\cong Z_{p^{2\lambda}}$ and $\ms{H}(p^{\lambda})$ exists by Theorem~\ref{thm:cyclic}, or $H\cong Z_{p^{\alpha}}\otimes Z_{p^{\beta}}$  for $\alpha+\beta=2\lambda$ and $\ms{H}(p^{\lambda})$ exists by Lemma~\ref{lem: Z_p^c + Z_p^c+d}.
	 If 
{$\Gamma=H$} 
 	we are done, otherwise 
{$|K|>1$ is odd} and
	$K \in \mathcal{G}$. Then by Theorem~\ref{thm:Kotzig} there exists a {$K$-Kotzig array}  of size $p^{\lambda}\times |K|$ and  we apply Lemma~\ref{lem:SWL}.
\end{proof}

%% file: 60_construction_n=2_s_22a.tex

\subsection{Construction for $n=2^s$}\label{sec:n=2^s}
In this section we deal with the case when $n$ is a power of $2$.

\begin{theorem}\label{thm:order 2^s}
	Let $|\Gamma| = 2^{2\alpha}$ and $\alpha\geq2$. Then there exists an $\gms(2^{\alpha})$.
\end{theorem}

\begin{proof}

	If $\Gamma$ is cyclic, then we are done by Theorem~\ref{thm:cyclic}. 

Assume that $\Gamma\not \cong Z_{2^{2\alpha}}$. The proof will be by induction on $|\Gamma|$.
Let {$\alpha=2$}, 
{then the $\ms{\Gamma}(4)$ for 
$\Gamma\in\{Z_2\oplus Z_8,Z_4\oplus Z_4,Z_2\oplus Z_2\oplus Z_4,Z_2\oplus Z_2\oplus Z_2\oplus Z_2\}$ are shown in Figures~\ref{fig: Z_2 x Z_8 and Z_4 x Z_4}--\ref{fig: Z_2^4}. }

\vskip10pt 

\begin{figure}[H]
\begin{subfigure}[t]{0.5\textwidth}
\begin{center}
\begin{tabular}{|c|c|c|c|}
\hline
    (0,0)  & (0,1) & (0,3) & (0,2)\\ \hline
    (0,7)  & (0,6) & (0,4) & (0,5)\\ \hline
    (1,0)  & (1,1) & (1,3) & (1,2)\\ \hline
    (1,7)  & (1,6) & (1,4) & (1,5)\\ \hline
  \end{tabular}
\subcaption{$\ms{Z_2\oplus Z_8}$}
\end{center}
\end{subfigure}
\hfill
\begin{subfigure}[t]{0.5\textwidth}
\begin{center}
\begin{tabular}{|c|c|c|c|}
\hline
    (1,1)  & (0,2) & (3,3) & (2,0)\\ \hline
    (0,3)  & (1,0) & (2,1) & (3,2)\\ \hline
    (2,2)  & (3,1) & (0,0) & (1,3)\\ \hline
	(3,0)  & (2,3) & (1,2) & (0,1)\\ \hline	
  \end{tabular} 
\subcaption{$\ms{Z_4\oplus Z_4}(4)$}
\end{center}
\end{subfigure}
\caption{}
\label{fig: Z_2 x Z_8 and Z_4 x Z_4}
\end{figure}

\vskip10pt

\begin{figure}[H]
\begin{center}
\begin{tabular}{|c|c|c|c|}
\hline
    (0,0,3)  & (0,0,2) & (0,1,3) & (0,1,2)\\ \hline
    (0,0,0)  & (0,0,1) & (0,1,0) & (0,1,1)\\ \hline
    (1,0,0)  & (1,0,1) & (1,1,0) & (1,1,1)\\ \hline
    (1,0,3)  & (1,0,2) & (1,1,3) & (1,1,2)\\ \hline
\end{tabular} 
\end{center}
\caption{$\ms{Z_2\oplus Z_2\oplus Z_4}(4)$}
\label{fig: Z_2 x Z_2 x Z_4}
\end{figure}


\vskip10pt
\begin{figure}[H]
\begin{center}
\begin{tabular}{|c|c|c|c|}
\hline
    (0,0,0,0)  & (0,1,0,0) & (0,0,0,1) & (0,1,0,1)\\ \hline
    (1,1,0,0)  & (1,0,0,0) & (1,1,0,1) & (1,0,0,1)\\ \hline
    (0,0,1,0)  & (0,1,1,0) & (0,0,1,1) & (0,1,1,1)\\ \hline
    (1,1,1,0)  & (1,0,1,0) & (1,1,1,1) & (1,0,1,1)\\ \hline
\end{tabular} 
\end{center}
\caption{$\ms{Z_2\oplus Z_2\oplus Z_2\oplus Z_2}(4)$}
\label{fig: Z_2^4}
\end{figure}

\vskip10pt
  Thus assume that $\alpha\geq 3$. Let $\Gamma=Z_{2^{\beta_1}}\oplus Z_{2^{\beta_2}}\oplus\dots\oplus Z_{2^{\beta_t}}$ and $\beta_1\leq\beta_2\leq\dots\leq\beta_t$.

	When $\beta_1=\beta_2=1$ or $\beta_1=2$ or $\beta_2=2$  we have $\Gamma=K\oplus H$ for $|H|=4$.  Because $\alpha\geq3$, we have {$|K|\geq 2^{4}$} and there exists a $\ms{K}(2^{\alpha-1})$ by inductive hypothesis and an $H$-Kotzig array of size $2^{\alpha-1}\times 4$ based on Theorem~\ref{thm:Kotzig}. The existence of $\gms(2^\alpha)$ then follows from Lemma~\ref{lem:SWL}.

	Assume now that $\beta_1>2$ is even. Thus $\beta_1=2\delta\geq 4$ and $\Gamma=K\oplus H$ where $K=Z_{2^{2\delta}}$. Therefore there exists a $\ms{K}(2^{\delta})$ by Theorem~\ref{thm:msquare} and  there exists an $H$-Kotzig array of size $2^{\delta}\times |H|$ by Theorem~\ref{thm:Kotzig}. The existence of $\gms(2^\alpha)$ then follows from Lemma~\ref{lem:SWL}.

	When $\beta_1\geq 1$ is odd, there must be a $\beta_{i}$ that is also odd and greater than one. Therefore  $\Gamma=K\oplus H$ where $K=Z_{2^{\beta_1}}\oplus Z_{2^{\beta_i}}$. Then there is an $\ms{K}(2^{(\beta_1+\beta_2)/2})$ by {Lemma~\ref{lem: Z_p^c + Z_p^c+d}}. If $t=2$, then we are done since $|H|=1$, for $t>1$ there is $|H|$ even and  there exists an $H$-Kotzig array of size $2^{(\beta_1+\beta_2)/2}\times |H|$ by Theorem~\ref{thm:Kotzig}. The existence of $\gms(2^\alpha)$ again follows from Lemma~\ref{lem:SWL}. This finishes the proof.
\end{proof}

%% file: 70_construction_n_even_22e.tex

\subsection{Construction for $n$ even, $n\neq 2^s$}\label{sec:n even}

In this section, we consider the case when $n$ is even but not a power of two. 
We start with the case when $\Gamma\cong Z_{4n}\oplus Z_n$ for $n$ odd.

\begin{lemma}\label{lem:Z_4n + Z_n}
Let $\Gamma\cong Z_{4n}\oplus Z_n$ for $n\geq 3$ odd. Then there exists a $\Gamma$-magic square $\ms{\Gamma}(2n)$.
\end{lemma}

\begin{proof}
We will show the construction of a $\Gamma$-magic square $\ms{\Gamma}(2n)$ in several steps.

\textbf{STEP 1.} Put elements of {$Z_{4n}$} in two rows such that, the sum in each column is 
{${4n-1}$}:
$$\begin{bmatrix}0&1&\ldots&n-1&n&n+1&\ldots&2n-2&2n-1\\
4n-1&4n-2&\ldots&3n&3n-1&3n-2&\ldots&2n+1&2n
\end{bmatrix}$$
Note that in this case each the row sum is $n$.

\textbf{STEP 2.} Exchange the element $n$ with $3n-1$ and $2n$ with $2n-1$ in the array:
$$R=\begin{bmatrix}0&1&\ldots&n-1&3n-1&n+1&\ldots&2n-2&2n\\
4n-1&4n-2&\ldots&3n&n&3n-2&\ldots&2n+1&2n-1
\end{bmatrix}$$
Note that now the sum in the first row is $n-n+3n-1-(2n-1)+2n=3n$, and in the second: $n+n-(3n-1)+(2n-1)-2n=-n\equiv 3n \pmod{4n}$.
Let $r_{i,j}$ be the element of this array in the $i$-th row and $j$-th column, for $i=1,2$ and $j=1,2,\ldots,2n$.

\textbf{STEP 3.} We will built an array $A$ of size $2n\times 2n$, such that

$$a_{i,j}=\begin{cases}(r_{1,j}, i-1),&i=1,2,\ldots,n \\
(r_{2,j}, i-1-n),&i=n+1,n+2,\ldots,2n \end{cases}$$
for $i,j\in\{1,2,\ldots,n\}$. Note that $a_{i,j} \in Z_{4n}\oplus Z_n$. 

\begin{adjustwidth}{-87.5pt}{}
$$\begin{bmatrix}(0,0)&(1,0)&\ldots&(n-1,0)&(3n-1,0)&(n+1,0)&\ldots&(2n-2,0)&(2n,0)\\
(0,1)&(1,1)&\ldots&(n-1,1)&(3n-1,1)&(n+1,1)&\ldots&(2n-2,1)&(2n,1)\\
\vdots&\vdots&\ldots&\vdots&\vdots&\vdots&\ldots&\vdots&\vdots\\
(0,n-1)&(1,n-1)&\ldots&(n-1,n-1)&(3n-1,n-1)&(n+1,n-1)&\ldots&(2n-2,n-1)&(2n,n-1)\\
(4n-1,0)&(4n-2,0)&\ldots&(3n,0)&(n,0)&(3n-2,0)&\ldots&(2n+1,0)&(2n-1,0)\\
(4n-1,1)&(4n-2,1)&\ldots&(3n,1)&(n,1)&(3n-2,1)&\ldots&(2n+1,1)&(2n-1,1)\\
\vdots&\vdots&\ldots&\vdots&\vdots&\vdots&\ldots&\vdots&\vdots\\
(4n-1,n-1)&(4n-2,n-1)&\ldots&(3n,n-1)&(n,n-1)&(3n-2,n-1)&\ldots&(2n+1,n-1)&(2n-1,n-1)\\
\end{bmatrix}$$
\end{adjustwidth}

	Now we need to check the row, column, and diagonal sums. For row $i$, we have
	\begin{align*}
	\rho_i 	&= \sum_{j=1}^{2m}a_{i,j}=(3n,2ni)=(3n,0).
	\end{align*}
	Similarly, for column $v$ we have
	\begin{align*}
	\sigma_j&= \sum_{i=1}^{2m}a_{i,j}=(-n,2\sum_{i=0}^{n-1}i)=(3n,0).
	\end{align*}
	For the main diagonal we obtain
	\begin{align*}
	\delta	&= \sum_{i=1}^{2n}a_{i,i}= (\sum_{i=0}^{n}i+\sum_{i=2}^{n-1}(3n-i)+2n-1,2\sum_{i=0}^{n-1}i)=(n-n^2,0).
	\end{align*}
	For the backward diagonal we obtain
	\begin{align*}
	\eta	&= \sum_{i=1}^{2n}a_{i,n+1-i}= (\sum_{i=1}^{n+1}(4n-i)+\sum_{i=1}^{n-2}(n+i)+2n,2\sum_{i=0}^{n-1}i)=(n+n^2,0).
	\end{align*}

\textbf{STEP 4.} We will permute some elements of the array $A$ of size $2n\times 2n$ to obtain a $\ms\Gamma(2n)=\{m_{i,j}\}_{i,j=1}^{2n}$.
Note that for $n\equiv 1 \pmod 4$, there is $\delta=(0,0)$ and $\eta=(2n,0)$, whereas for $n\equiv 3 \pmod 4$, there is $\delta=(2n,0)$ and $\eta=(0,0)$.

\newpage
Assume first that $n\equiv 1 \pmod 4$.
Let 
\begin{align*}
	&m_{n+1,n+1}		&&=a_{n+1,n}	&&=(3n,0),\\ 
	&m_{n+1,n}			&&=a_{n+1,n+1}	&&=(n,0),\\  
	&m_{1,n+1}			&&=a_{1,n}		&&=(n-1,0),\\ 
	&m_{1,n}			&&=a_{1,n+1}	&&=(3n-1,0),\\
	&m_{(n+1)/2,(n+1)/2}
		&&=a_{(n+1)/2,(3n+1)/2}			&&=((3n-1)/2,(n-1)/2),\\ 
	&m_{(n+1)/2,(3n+1)/2}
		&&=a_{(n+1)/2,(n+1)/2}			&&=((n-1)/2,(n-1)/2),\\  
	&m_{n+1,(n+1)/2}
		&&=a_{n+1,(3n+1)/2}				&&=((5n-1)/2,0),\\ 
	&m_{n+1,(3n-1)/2}	
		&&=a_{n+1,(n+1)/2}				&&=((7n-1)/2,0).
\end{align*}
	{For all other entries not appearing in the above list, we set $m_{i,j}=a_{i,j}.$}

	Observe that now all row, column, diagonal and backward diagonal sums are $(3n,0)$.

Assume now that $n\equiv 3 \pmod 4$.
Let 
\begin{align*}
	&m_{(n+1)/2,(n+1)/2}	&&=a_{(n+1)/2,(3n+1)/2}	&&=((3n-1)/2,(n-1)/2),\\
	&m_{(n+1)/2,(3n+1)/2}	&&=a_{(n+1)/2,(n+1)/2}	&&=((n-1)/2,(n-1)/2),\\
	&m_{n+1,(n+1)/2}		&&=a_{n+1,(3n+1)/2}		&&=((5n-1)/2,0),\\
	&m_{n+1,(3n+1)/2}		&&=a_{n+1,(n+1)/2}		&&=((7n-1)/2,0)
\end{align*}
	{For all other entries not appearing in the above list, we again set $m_{i,j}=a_{i,j}.$}

	Observe that now all row, column, diagonal and backward diagonal sums are $(3n,0)$.
\end{proof}

\begin{exm}\label{exm: Z_12 x Z_3}
In Figure~\ref{fig:example Z_12 x Z_3} we show an $\text{MS}_{ Z_{12}\oplus Z_3}(6)$. 
\end{exm}

\begin{figure}[h]
\begin{center}
$
\begin{array}{|c|c|c|c|c|c|}
	\hline
	(0,0)& (1,0)& (2,0)& (8,0)& (4,0)& (6,0)\\\hline
	(0,1)& (4,1)& (2,1)& (8,1)& (1,1)& (6,1)\\\hline
	(0,2)& (1,2)& (2,2)& (8,2)& (4,2)& (6,2)\\\hline
	(11,0)& (7,0)& (9,0)& (3,0)& (10,0)& (5,0)\\\hline
	(11,1)& (10,1)& (9,1)& (3,1)& (7,1)& (5,1)\\\hline
	(11,2)& (10,2)& (9,2)& (3,2)& (7,2)& (5,2)\\\hline
\end{array}
$
\end{center}	
\caption{${\ms{Z_{12}}\oplus Z_3}(6)$}
\label{fig:example Z_12 x Z_3}
\end{figure}

The proof of the lemma below is similar to the proof of Lemma~\ref{lem:SWL}, but instead of a Kotzig array on an Abelian group, we will use  a regular Kotzig array.


\begin{lemma}\label{lem:gl2}
Let $\Gamma\cong Z_{4n^{\alpha}}\oplus Z_{n^{\beta}}$  where $n>1$ {is odd}, $\alpha,\beta\geq 1$ and $\alpha+\beta$ even. Then there exists a $\Gamma$-magic square $\ms{\Gamma}(2n^{(\alpha+\beta)/2})$.
\end{lemma}


\begin{proof}
	{Without loss of generality we can assume $\alpha\leq \beta$. The proof will be by induction on $\beta$. If  $\alpha=\beta$,   then there exists 
{$\ms{Z_{4n^{\alpha}}
	\oplus Z_{n^{\alpha}}}
	(2n^{\alpha})$} by Lemma~\ref{lem:Z_4n + Z_n}. From now assume that $\beta>\alpha$.} 
	If $\alpha$ is even, then $\beta$ is even and there exists a $\ms{Z_{4n^{\alpha}}}(2n^{\alpha/2})$ by {Theorem~\ref{thm:cyclic}}. The result then follows by Lemma~\ref{lem:SWL}.
	Assume now  that $\alpha\geq 1$ is odd, then $\beta\geq 3$ is odd.  Let $K=\left\langle  n^2\right\rangle$ be a subgroup of $Z_{n^{\beta}}$. 
	Define  $\Gamma_0=Z_{4n^{\alpha}}\oplus K$. Let $m=n^{(\alpha+\beta-2)/2}$, then by the inductive hypothesis there exists $\ms{\Gamma_0 }(2m)$ with a magic constant $\delta_0$, because 
	$\Gamma_0=Z_{4n^{\alpha}}\oplus K \cong Z_{4n^{\alpha}}\oplus Z_{n^{\beta-2}}$. Denote  by $y_{i,j}=(y_{i,j}^{(1)},y_{i,j}^{(2)})$  the element in the $i$-th row and $j$-th column of the $\Gamma_0$-magic square.

We will build $n^2$ different \emph{residual squares} $R^s$ of size $2m\times 2m$  with entries 
{$r^s_{i,j}\in \{0,1,\ldots,n^2-1\}$}
	for $1 \leq s \leq n^2$.  
	$$r^s_{i,j}=\left\{
\begin{array}{lcl}
s-1, & \text{if} & i,j \;\text{odd}, \\
n^2-s, & \text{if} & i \;\text{odd},\;\;j\;\;\text{even}\\
n^2-1-r_{i-1,j}, & \text{if} & i \;\text{even}, \\
\end{array}%
\right.
$$ where $1\leq i,j\leq 2m$. 
 Observe that all column, row sums are equal to $mn^2$, the diagonal sum is $2m(s-1)$, and backward diagonal is $2m(n^2-s)$.

We will construct now $n^2$ squares $X^1,\ldots,X^{n^2}$ of size $2m\times 2m$. We will glue those $n^2$ squares together to obtain an MS$_\Gamma(2mn)$
Denote by $x^s_{i,j}$ the entry in the $i$-th row and $j$-th column of the $s$-th square $X^s$.
Recall that $\Gamma=Z_{4n^{\alpha}}\oplus Z_{n^{\beta}}$, therefore any $g\in \Gamma$ we can identify with a pair $(a,b)$ such that $a\in Z_{4n^{\alpha}}$, $b\in Z_{n^{\beta}}$.  Let $x_{i,j}^s=(y_{i,j}^{(1)},y_{i,j}^{(2)}+r^s_{i,j})$, for $i,j=1,2,\ldots,2m$ and $s=1,2,\ldots,n^2$.

Suppose that $x_{i,j}^s= x_{i',j'}^{s'}$. Then $(y_{i,j}^{(1)},y_{i,j}^{(2)}+r^s_{i,j})=(y_{i',j'}^{(1)},y_{i',j'}^{(2)}+r^{s'}_{i',j'})$, which implies    $y_{i,j}^{(2)}+r^s_{i,j}=y_{i',j'}^{(2)}+r^{s'}_{i,j}$, but then $0\leq r^s_{i,j}-r^{s'}_{i',j'}=y_{i',j'}^{(2)}-y_{i,j}^{(2)}=k n^2$ for an integer $k\neq 0$. This is impossible because $r^s_{i,j},r^{s'}_{i',j'}\in\{0,1,\dots,n^2-1\}$. If $k=0$, then $i=i'$, $j=j'$ and thus $s=s'$, a contradiction.

Obviously all row and column sums are equal to a
constant $\delta=\delta_0+(0,mn^2)$. {Observe that, because $n$ is odd, there exists exactly one $s \in Z_{n^{\beta}}$ such that $2s\equiv n^2 \pmod{n^{\beta}}$. For  $X^s$ such that $2s\not\equiv n^2 \pmod{n^{\beta}}$} we set $s'= n^2 -s$ and call $X^{s'}$ the \emph{complementary square} of $X^s$. Note that $2ms+2ms'\equiv 2mn^2 \pmod {n^{\beta}}$. For $n$ odd we will glue now the rectangles $X^1,X^2,\ldots,X^{n^2}$ in such way that the square $X^{0}$ is in the center (i.e., it is the square which has elements on both the main and backward main diagonal) and  for any other square $X^s$ lying on a diagonal  there is also its complementary square on this diagonal.
\end{proof}


%% file: 80_main_22h.tex

\subsection{Main result} \label{sec:main result}
 We summarize our previous results a single theorem.

\begin{theorem}\label{thm:main}
	{Let $\Gamma$ be an Abelian group of order $n^2>1$. There exists a $\Gamma$-magic square $\gms(n)$  if and only if $n>2$.}
\end{theorem}

\begin{proof}
	{
	The necessary condition follows from Observation~\ref{obs:necessary}.
	} 
	For clarity, we split the proof of sufficiency into several cases.
	
\vskip6pt\noindent
\textit{\underline{Case 1} \ $n$ is odd}	
\begin{adjustwidth}{40pt}{0pt}
	When $n$ is odd, the proof follows directly from Theorem~\ref{thm:main-odd}. 
\end{adjustwidth}

\vskip6pt\noindent
\textit{\underline{Case 2} \ $n=2^{2\beta}$}	
\begin{adjustwidth}{40pt}{0pt}
	When $n=2^{{2}\beta}$, the result follows from Theorem~\ref{thm:order 2^s}. 
\end{adjustwidth}

\vskip6pt\noindent
\textit{\underline{Case 3} \ $n=4k^2, \ k> 1$}	
\begin{adjustwidth}{40pt}{0pt}
	When  $n=4k^2$ where $k$ is odd and $k>1$, then either $\Gamma=Z_2\oplus Z_2\oplus K$ or $\Gamma=Z_4\oplus K$ with $|K|=k^2$ odd.

    If $\Gamma=Z_2\oplus Z_2\oplus K$, then there exists $\ms{K}(k)$ by {Case 1} and since there exists a $Z_2\oplus Z_2$-Kotzig array of size $k\times 4$ by  Theorem~\ref{thm:Kotzig},  we can apply Lemma~\ref{lem:SWL}.
      
	If $\Gamma=Z_4\oplus K$ and $K$ is cyclic, then $\Gamma$ is cyclic as well and the result follows from Theorem~\ref{thm:cyclic}.

	If $K$ is not cyclic, then there exists an odd prime $p$ such that 
	$K=Z_{p^\gamma}\oplus Z_{p^\delta}\oplus H$ where $\gamma+\delta$ is even and we can write
	$\Gamma=Z_{4p^\gamma}\oplus Z_{p^\delta}\oplus H$. Then there exists a $Z_{4p^\gamma}\oplus Z_{p^\delta}$-magic square of side $(\gamma+\delta)/2$ by Lemma~\ref{lem:gl2}. When $|H|=1$, then $\Gamma=Z_{4p^\gamma}\oplus Z_{p^\delta}$ and we are done. When $|H|>1$, then we apply Lemma~\ref{lem:SWL}.
\end{adjustwidth}

\vskip6pt\noindent
\textit{\underline{Case 4} \ $n=2^{2\beta}k^2, \ k> 1, \beta>1$}	
\begin{adjustwidth}{40pt}{0pt}
	Finally, when $\beta\geq2$, then $\Gamma=K\oplus H$ where $|K|=2^{2\beta}$. Then there is an $\ms{K}(2^\beta)$ by Theorem~\ref{thm:order 2^s}  and an $H$-Kotzig array of size $2^\beta\times |H|$ by Theorem~\ref{thm:Kotzig}. Hence we apply again Lemma~\ref{lem:SWL}.	
\end{adjustwidth}
\vskip-14pt
\end{proof}

%% file: 90_conclusion_22d.tex
\newpage
\section{Conclusion} \label{sec:conclusion}

We have shown that a $\Gamma$-magic square of side $n$ over a finite Abelian group $\Gamma$ exists for every $n>2$ and every $\Gamma$ of order $n^2$. 

In our paper, we construct a magic square $\gms(n)$, where $\Gamma$  is an Abelian group of order $n^2$. A natural question arises if there exists a magic {square} $\gms(n)$ when $\Gamma$ is an Abelian group of order $k>n^2$.

Formally speaking, let $(\Gamma,+)$ be an Abelian group of order $k>n^2$, $S\subset \Gamma$ such that $|S|=n^2$  and MS$_{S}(n)$ be an {$n\times n$} array whose entries are all elements of $S$. Then MS$_{S}(n)$ is an $S$-magic {square}  if all row, column, main and backward main diagonal sums are equal to the same element $\mu\in\Gamma$.

We propose the following open problem:

\begin{oprb}Given an integer $n$ and an Abelian group $\Gamma$ of order  $k>n^2$, determine whether there exists MS$_{S}(n)$ for some $S\subset \Gamma$ {such that $S$ is not a subgroup of $\Gamma$}.
\end{oprb}

	Since for every $n>2$ and a group $Z_{n^2+1}$ with 	$S=Z_{n^2+1}\setminus\{0\}$ there exists MS$_{S}(n)$ by  Theorem~\ref{thm:cyclic}, we can state a more specific version of this problem:
\begin{oprb}Given an integer $n$ and an Abelian group $\Gamma$ of order  $n^2+1$, determine whether there exists MS$_{S}(n)$ for $S= \Gamma\setminus \{0\}$.
\end{oprb}

While the order in which edge labels are considered is irrelevant for Abelian groups, in non-Abelian groups, different orders may yield different weights. Therefore, one could  also inquire about the existence of magic squares  in non-Abelian groups.

\section{Statements and Declarations}
The work of the first author was  supported by program ''Excellence initiative – research university'' for the AGH University.

%% file: 99_references_22a.tex
\vskip1cm

\noindent

%% file: MS.bbl
\newpage

\begin{thebibliography}{00}

\bibitem{CicZ}  S. Cichacz,  {Zero sum partition into sets of the same order and its applications},   \textit{Electron. J.
Combin.} 25(1) (2018), \#P1.20.



\bibitem{Cichacz-Hincz-2}
S. Cichacz, T. Hinc,
A magic rectangle set on Abelian groups and its application,
\textit{Discrete Appl. Math.} \textbf{288} (2021), 201--210.


\bibitem{handbook}
C.J. Colbourn, J.H. Dinitz, eds.,
\textit{Handbook of combinatorial designs}, second edn.,
\textsc{Discrete Mathematics and its Applications (Boca Raton)}, 
Chapman \& Hill/CRC Press, Boca Raton, FL 2007.






\bibitem{Evans}
A.B. Evans, 
Magic rectangles and modular magic rectangles, 
\textit{J. Stat. Plann. Inference} \textbf{51} (1996), 171--180.





\bibitem{Heinrich-Hilton}
K. Heinrich, A.J.W. Hilton, 
Doubly diagonal orthogonal Latin squares, 
\textit{Discrete Math.}, \textbf{46} (1983), 173--182.






\bibitem{M-W}
A.M. Marr, W.D. Wallis,
\textit{Magic graphs}, second edn.,
Birkh\"auser, 2013.






\bibitem{Sun-Yihui}
H. Sun, W. Yihui, 
Note on magic squares and magic cubes on Abelian groups, 
\textit{J. Math. Res. Exposition} \textbf{17(2)} (1997), 176--178.



\end{thebibliography}
